\documentclass[11pt]{amsart}

\usepackage{amssymb,epstopdf,subcaption,mathtools,tikz-cd,pb-diagram,thmtools, thm-restate,pinlabel,pifont,stmaryrd}
\usepackage{graphicx}
\usepackage[shortlabels]{enumitem}
\usepackage[all]{xy}
\usepackage[mathscr]{euscript}
 \usepackage{pinlabel, float}
\usepackage{hyperref}

\usepackage{caption}
\usepackage{subcaption}

\newcommand{\nc}{\newcommand}
\nc{\dmo}{\DeclareMathOperator}
\nc{\nt}{\newtheorem}

\newtheorem{theorem}{Theorem}[section]
\newtheorem{mainthm}{Theorem}

\newtheorem{lemma}[theorem]{Lemma}
\newtheorem{proposition}[theorem]{Proposition}
\newtheorem{corollary}[theorem]{Corollary}
\theoremstyle{definition}

\newtheorem{question}[theorem]{Question} 
\newtheorem{problem}[theorem]{Problem} 
\newtheorem{claim}[theorem]{Claim}
\newtheorem{remark}[theorem]{Remark}
\theoremstyle{remark}

\nc{\cut}{\!\ssearrow\!}

\dmo{\Diff}{Diff}
\dmo{\Mod}{Mod}
\dmo{\SMod}{SMod}
\dmo{\I}{\mathcal{I}}
\dmo{\SO}{SO}
\dmo{\Orth}{O}
\dmo{\Sp}{Sp}
\dmo{\SL}{SL}
\dmo{\GL}{GL}
\dmo{\im}{im}

\dmo{\Emb}{Emb}
\dmo{\PSp}{PSp}
\dmo{\PSL}{PSL}
\dmo{\PMod}{PMod}
\dmo{\Homeo}{Homeo}
\dmo{\Twist}{Twist}
\dmo{\Aut}{Aut}
\dmo{\Nil}{Nil}
\dmo{\Sol}{Sol}
\dmo{\Isom}{Isom}
\dmo{\Out}{Out}
\dmo{\OUT}{\mathcal{OUT}}
\dmo{\AUT}{\mathcal{AUT}}
\dmo{\Inn}{Inn}
\dmo{\ab}{ab}
\dmo{\id}{id}
\dmo{\orb}{orb}
\dmo{\Hom}{Hom}
\dmo{\Fix}{Fix}
\dmo{\Ind}{Ind}
\dmo{\Stab}{Stab}
\dmo{\pt}{pt}

\nc{\car}{\curvearrowright}
\nc{\what}{\widehat}
\nc{\Si}{\Sigma}
\nc{\D}{\mathbb D}
\nc{\Z}{\mathbb Z}
\nc{\N}{\mathcal N}
\nc{\R}{\mathbb R}
\nc{\F}{\mathcal F}
\nc{\C}{\mathbb{C}}
\nc{\onto}{\twoheadrightarrow}
\nc{\cc}{\mathcal C}
\nc{\B}{\mathcal{B}}

\nc{\ga}{\gamma}
\nc{\de}{\delta}
\nc{\ep}{\epsilon}
\nc{\ra}{\rightarrow}
\nc{\ti}{\times}
\nc{\wtil}{\widetilde}
\nc{\pa}{\partial}
\nc{\pair}[1]{\langle #1 \rangle}

\nc{\flm}{\lambda_{2}}

\nc{\normalclosure}[1]{\ensuremath{\left \langle \left \langle #1 \right \rangle \right \rangle}}

\nc{\margin}[1]{\marginpar{\scriptsize #1}}

\nc{\p}[1]{\bigskip\noindent\textbf{#1.}}
\nc{\lei}[1]{{\color{red} \sf  L: [#1]}}
\nc{\bena}[1]{{\color{blue} \sf  B: [#1]}}

\title{Surface mapping class group actions on 3-manifolds}
\author{Alina al Beaini} 
\author{Lei Chen}
\author{Bena Tshishiku}

\address{Alina Al Beaini \\ Department of Mathematics\\ Brown University \\ 151 Thayer St.  \\ Providence, RI, 02912, USA\\  alina$\_$al$\_$beaini@brown.edu.   }

\address{Lei Chen \\ Department of Mathematics\\ University of Maryland \\ 4176 Campus Drive\\ College Park, MD 20742, USA \\  chenlei@umd.edu}

\address{Bena Tshishiku \\ Department of Mathematics\\ Brown University \\ 151 Thayer St.  \\ Providence, RI, 02912, USA\\  bena$\_$tshishiku@brown.edu.   }

\date{\today}

\begin{document}

\maketitle

\vspace*{-4ex}

\begin{abstract}
For each circle bundle $S^1\to X\to\Sigma_g$ over a surface with genus $g\ge2$, there is a natural surjection $\pi:\Homeo^+(X)\to\Mod(\Sigma_g)$. When $X$ is the unit tangent bundle $U\Sigma_g$, it is well-known that $\pi$ splits. On the other hand $\pi$ does not split when the Euler number $e(X)$ is not divisible by the Euler characteristic $\chi(\Sigma_g)$ by \cite{mod-circle-bundle}. In this paper we show that this homomorphism does not split in many cases where $\chi(\Si_g)$ divides $e(X)$. 
\end{abstract}

\section{Introduction}

Let $\Si_g$ be a closed oriented surface of genus $g\ge2$, and let $X_{g,e}$ denote the oriented $S^1$-bundle over $\Si_g$ with Euler number $e$. Let $\Homeo^+(X_{g,e})$ be the group of orientation-preserving homeomorphisms of $X_{g,e}$ that act trivially on the center of $\pi_1(X_{g,e})$, and let $\Mod(X_{g,e}):=\pi_0\big(\Homeo^+(X_{g,e})\big)$ denote the mapping class group. 

The (generalized) Nielsen realization problem for $X_{g,e}$ asks whether the surjective homomorphism 
\[\Homeo^+(X_{g,e})\to \Mod(X_{g,e})\]
splits over subgroups of $\Mod(X_{g,e})$. In this paper we study a closely related problem. For each $g,e$ there is a surjection $\Mod(X_{g,e})\to\Mod(\Si_g)$. Consider the composition
\[\pi_{g,e}: \Homeo^+(X_{g,e})\onto \Mod(X_{g,e})\onto\Mod(\Si_g).\]

\begin{problem}\label{prob:main}
Does $\pi_{g,e}:\Homeo^+(X_{g,e})\onto\Mod(\Si_g)$ spilt? 
\end{problem}

If $e=\pm(2g-2)$, then $X_{g,e}$ is the unit (co)tangent bundle, and $\pi_{g,e}$ does split; see \cite[\S1]{Souto}. On the other hand, if $e$ is not divisible by $2g-2$, then the surjection $\Mod(X_{g,e})\to\Mod(\Si_g)$ does not split by work of the second two authors \cite{mod-circle-bundle}, so $\pi_{g,e}$ also does not split in these cases. Given this, it remains to study the case when $e$ is divisible by $2g-2$ and $e\neq\pm(2g-2)$. In these cases $\Mod(X_{g,e})\to\Mod(\Si_g)$ does split \cite{mod-circle-bundle}, but we prove $\pi_{g,e}$ does not split in many cases. 




\begin{mainthm}\label{thm:main} Fix a surface $\Sigma_g$ of genus $g$ and $e\in\mathbb Z$. 
Assume that $g=4k-1$ where $k\ge3$ and $k$ is not a power of $2$, and assume that $e$ is divisible by $(2g-2)2p$ where $p$ is an odd prime dividing $k$. Then the natural surjective homomorphism $\pi_{g,e}:\Homeo^+(X_{g,e})\to\Mod(\Si_{g})$ does not split.
\end{mainthm}

For example, if $e=0$, we find that $\pi_{g,e}:\Homeo^+(\Si_g\times S^1)\to\Mod(\Si_g)$ does not split when $g=11, 19, 23, 27, 35, 39, 43, 47, \ldots$. 

Theorem \ref{thm:main} solves the Nielsen realization problem for $\Mod(\Si_g)$ subgroups of $\Mod(X_{g,e})$ in the cases of the theorem. Specifically, if $e$ is divisible by $2g-2$, then $\Mod(X_{g,e})\cong H^1(\Si_g;\Z)\rtimes\Mod(\Si_g)$ \cite{mod-circle-bundle}, and every $\Mod(\Si_g)$ subgroup of $\Mod(X_{g,e})$ is the image of a splitting of $\Mod(X_{g,e})\to\Mod(\Si_g)$. By Theorem \ref{thm:main}, $\Homeo^+(X_{g,e})\to\Mod(X_{g,e})$ does not split over any of these $\Mod(\Si_g)$ subgroups. 

Theorem \ref{thm:main} has the following topological consequence. When $2g-2$ divides $e$, there is a ``tautological" $X_{g,e}$-bundle $E_{g,e}^{\text{taut}} \to B\Homeo(\Si_g)$ whose monodromy 
\[\Mod(\Si_g)\cong\pi_1\big(B\Homeo(\Si_g)\big)\to\Mod(X_{g,e})\] splits the surjection $\Mod(X_{g,e})\to\Mod(\Si_g)$ (c.f.\ \cite[\S1]{mod-circle-bundle}). One can ask whether or not the bundle $E_{g,e}^{\text{taut}}\to B\Homeo(\Si_g)$ is flat. Recall that an $X$-bundle $E\to B$  is \emph{flat} if there is a homomorphism $\rho:\pi_1(B)\to\Homeo(X)$ and an $X$-bundle isomorphism $E\cong X\rtimes_\rho B$. Such bundles are characterized by the existence of a horizontal foliation on $E$, or, equivalently, by the property that their monodromy $\pi_1(B)\to\Mod(X)$ lifts to $\Homeo(X)$. When $e=2g-2$, the bundle $E_{g,e}^{\text{taut}}\to B\Homeo(\Si_g)$ is flat because of the splitting of $\pi_{g,e}$ in this case. When $\pi_{g,e}$ does not split, we deduce that $E_{g,e}^{\text{taut}}\to B\Homeo(S_g)$ is not flat.

\begin{corollary}
Fix $g,e$ as in the statement of Theorem \ref{thm:main}. Then the tautological $X_{g,e}$-bundle $E_{g,e}^{\text{taut}}\to B\Homeo(S_g)$ is not flat. 
\end{corollary}



\p{Short proof sketch of Theorem \ref{thm:main}} The proof strategy is similar to an argument of Chen--Salter \cite{Chen-Salter} that shows that $\Homeo^+(\Si_g)\to\Mod(\Sigma_g)$ does not split when $g\ge2$. Theorem \ref{thm:main} is proved by contradiction: assuming the existence of a splitting $\Mod(\Sigma_g)\to\Homeo^+(X_{g,e})$, first we obtain, by lifting, an action of the based mapping class group $\Mod(\Sigma_{g},*)$ on the cover $\what X_{g,e}\cong\mathbb R^2\times S^1$ corresponding to the center of $\pi_1(X_{g,e})$. The conditions on $g$ and $e$ in Theorem \ref{thm:main} guarantee the existence of a $\Z/2p\Z$ subgroup of $\Mod(\Sigma_g,*)$ for which we can show the action on $\what X_{g,e}$ has a fixed circle. Denoting a generator of $\Z/2p\Z$ by $\alpha$, we show that $\Mod(\Sigma_{g},*)$ is generated by the centralizers of $\alpha^2$ and $\alpha^p$. This shows that the entire group $\Mod(\Sigma_{g},*)$ acts on $\what X_{g,e}$ with a fixed circle, which contradicts the fact that the point-pushing subgroup $\pi_1(\Sigma_{g})<\Mod(\Sigma_{g},*)$ acts freely (by deck transformations) on $\what X_{g,e}$. 

\p{Other questions} 
Related to the $\Mod(\Si_g)$ action on the unit tangent bundle $U\Si_g$, we pose the following question. 

\begin{question}Do either of the following surjections split?  
\[\Diff^+(U\Si_g)\to\Mod(\Si_g)\>\>\>\text{ or }\>\>\>\Homeo(U\Si_g)\to\Mod(U\Si_g)\] 
\end{question}

If one includes orientation-reversing diffeomorphisms and mapping classes, then if $g\ge12$, then $\Diff(U\Si_g)\to\Mod^{\pm}(\Si_g)$ does not split by Souto \cite[Thm.\ 1]{Souto}. 

\p{Acknowledgement} The authors LC and BT are supported by NSF grants DMS2203178, DMS-2104346, and DMS-2005409. 

\section{Proof of Theorem \ref{thm:main}} 

Fix $g=4k-1$ and $e$ as in the theorem statement, and set $\Sigma=\Sigma_g$ and $X=X_{g,e}$. Suppose for a contradiction that there is a homomorphism 
\[\sigma:\Mod(\Sigma)\to\Homeo(X)\] whose composition with $\pi=\pi_{g,e}:\Homeo(X)\to\Mod(\Si)$ is the identity.

\subsection{Step 1: lifting argument}\label{sec:lifting}

Consider the covering space $\widehat X=\widetilde\Sigma\times S^1$ of $X$, where $\widetilde\Sigma\cong\R^2$ is the universal cover. This is the covering corresponding to the center $\zeta$ of $\pi_1(X)$. Given the action of $\Mod(\Sigma)$ on $X$, we consider the set of all lifts of homeomorphisms in this action to $\what X$. This is an action of the pointed mapping class group $\Mod(\Sigma,*)$ on $\what X$. To explain this, we start with the following general proposition. 

\begin{proposition}\label{prop:diagram}
Let $Y$ be a closed manifold. Let $\zeta<\pi_1(Y)$ be the center of the fundamental group, and denote $\Delta=\pi_1(Y)/\zeta$. Let $\widehat Y\to Y$ be the covering space with $\pi_1(\widehat Y)=\zeta$. Fix a basepoint $*\in Y$. Assume that the evaluation map \[\Homeo(Y)\to Y, \>\>\>\>f\mapsto f(*)\] induces a surjection $\pi_1\big(\Homeo(Y)\big)\onto\zeta<\pi_1(Y)$. Then there is a commutative diagram
\begin{equation}\label{eqn:diagram}
\begin{xy}
(-60,0)*+{1}="A";
(-30,0)*+{\Delta}="B";
(0,0)*+{\Homeo(\widehat Y)^{\Delta}}="C";
(30,0)*+{\Homeo(Y)}="D";
(60,0)*+{1}="E";
(-60,-15)*+{1}="F";
(-30,-15)*+{\Delta}="G";
(0,-15)*+{\Mod(Y,*)}="H";
(30,-15)*+{\Mod(Y)}="I";
(60,-15)*+{1}="J";
{\ar"A";"B"}?*!/_3mm/{};
{\ar "B";"C"}?*!/_3mm/{};
{\ar "C";"D"}?*!/_3mm/{q};
{\ar "D";"E"}?*!/_3mm/{};
{\ar"F";"G"}?*!/_3mm/{};
{\ar "G";"H"}?*!/_3mm/{};
{\ar "H";"I"}?*!/_3mm/{};
{\ar "I";"J"}?*!/_3mm/{};
{\ar@{=} "B";"G"}?*!/_3mm/{};
{\ar "C";"H"}?*!/_3mm/{\widehat p};
{\ar "D";"I"}?*!/_3mm/{p};
\end{xy}\end{equation}
whose rows are exact, where the bottom row is the (generalized) Birman exact sequence. Furthermore, this diagram is a pullback diagram. 
\end{proposition}

A version of Proposition \ref{prop:diagram} when $Y=\Sigma_g$ (whose center is trivial) is used in \cite{Chen-Salter}. 

We prove Proposition \ref{prop:diagram} after explaining how it gives the desired lifting. In our situation, the center of $\pi_1(X)$ is the kernel of $\pi_1(X)\to\pi_1(\Sigma)$ since $\pi_1(\Sigma)$ has trivial center. Thus Proposition \ref{prop:diagram} gives us the following diagram. 
\[\begin{xy}
(-60,0)*+{1}="A";
(-30,0)*+{\pi_1(\Sigma)}="B";
(0,0)*+{\Homeo(\widehat X)^{\Delta}}="C";
(30,0)*+{\Homeo(X)}="D";
(60,0)*+{1}="E";
(-60,-15)*+{1}="F";
(-30,-15)*+{\pi_1(\Sigma)}="G";
(0,-15)*+{\Mod(X,*)}="H";
(30,-15)*+{\Mod(X)}="I";
(60,-15)*+{1}="J";
(-60,-30)*+{1}="K";
(-30,-30)*+{\pi_1(\Sigma)}="L";
(0,-30)*+{\Mod(\Sigma,*)}="M";
(30,-30)*+{\Mod(\Sigma)}="N";
(60,-30)*+{1}="O";
{\ar"A";"B"}?*!/_3mm/{};
{\ar "B";"C"}?*!/_3mm/{};
{\ar "C";"D"}?*!/_3mm/{q};
{\ar "D";"E"}?*!/_3mm/{};
{\ar"F";"G"}?*!/_3mm/{};
{\ar "G";"H"}?*!/_3mm/{};
{\ar "H";"I"}?*!/_3mm/{};
{\ar "I";"J"}?*!/_3mm/{};
{\ar"K";"L"}?*!/_3mm/{};
{\ar "L";"M"}?*!/_3mm/{};
{\ar "M";"N"}?*!/_3mm/{};
{\ar "N";"O"}?*!/_3mm/{};
{\ar@{=} "B";"G"}?*!/_3mm/{};
{\ar "C";"H"}?*!/_3mm/{\widehat p};
{\ar "D";"I"}?*!/_3mm/{p};
{\ar@{^{(}->} "N";"I"}?*!/_3mm/{};
{\ar@{^{(}->} "M";"H"}?*!/_3mm/{};
{\ar@{=} "L";"G"}?*!/_3mm/{};
\end{xy}\]

The splitting $\sigma$ defines a subgroup $\Mod(\Sigma)<\Mod(X)$ and a splitting of $p$ over this subgroup. Since the top row is a pullback of the middle row, it follows that $\widehat p$ splits over $\Mod(\Sigma,*)$ (this uses only general facts about pullbacks). Denote this splitting by 
\[\widehat\sigma:\Mod(\Sigma,*)\to\Homeo(\widehat X)^{\Delta}.\] Under this splitting the point-pushing subgroup $\pi_1(\Sigma)$ acts by deck transformations. 

\begin{remark}\label{rmk:lift-fp}
If $G<\Mod(\Sigma)$ and $\sigma(G)$ has a fixed point $*$, then after choosing a lift $\what*$ of $*$, one can lift canonically elements of $\sigma(G)$ to $\what X$ by choosing the unique lift that fixes $\what *$. This implies that $G<\Mod(\Sigma)$ can be lifted to $G<\Mod(\Sigma,*)$ so that $\what\sigma(G)$ has a fixed point. 
\end{remark}

\begin{proof}[Proof of Proposition \ref{prop:diagram}]
First we recall the construction of the bottom row of diagram (\ref{eqn:diagram}). Evaluation at $*\in Y$ defines a fibration 
\[\Homeo(Y,*)\to \Homeo(Y)\xrightarrow{\epsilon} Y.\] 
The long exact sequence of homotopy groups gives an exact sequence
\[\pi_1\big(\Homeo(Y))\xrightarrow{\epsilon_*}\pi_1(Y)\to\Mod(Y,*)\to\Mod(Y)\to1.\]
In general the image of $\epsilon_*$ is contained in the center of $\pi_1(Y)$; see e.g.\ \cite[\S1.1, Exer.\ 20]{hatcher-algtop}. By assumption, $\epsilon_*$ surjects onto the center, so we obtain the short exact sequence in the bottom row of (\ref{eqn:diagram}). The homomorphism $\pi_1(Y)\to\Mod(Y,*)$ is the so-called ``point-pushing" homomorphism. It sends $\eta\in\pi_1(Y)$ (basepoint= $*$) to the time-1 map of an isotopy that pushes $*$ around $\eta$ in reverse (this follows directly from the definition of the connecting homomorphism in the long exact sequence; note that it makes sense for the reverse of $\eta$ to appear in defining this homomorphism since concatenation of paths is left-to-right, while composition of functions is right-to-left). 

Next we define $\widehat p:\Homeo(\widehat Y)^\Delta\to\Mod(Y,*)$. Fix a point $\widehat{*}\in\widehat Y$ that covers the basepoint $*\in Y$. Given $f\in\Homeo(\widehat Y)^{\Delta}$. Choose a path $[0,1]\to\widehat Y$ from $\widehat *$ to $f(\widehat *)$ and let $\gamma_f$ denote the composition $[0,1]\to\widehat Y\to Y$. By isotopy extension, there exists an isotopy $h_t:Y\to Y$ where $h_0=\id_Y$ and $h_t(*)=\gamma_f(t)$ for each $t\in[0,1]$. Define 
\[\widehat p(f)=[h_1\circ q(f)].\]
The map $\widehat p$ is well-defined. The choice of $\gamma_f$ is unique only up to an element of $\pi_1(\widehat Y)=\zeta$. This implies that the isotopy class $[h_1\circ f]$ is only well-defined up to composition by a point-pushing mapping class by an element of $\zeta$, but such a point-push is trivial by assumption. 

It is a straightforward exercise to check that $\widehat p$ is a homomorphism. The right square in diagram (\ref{eqn:diagram}) commutes because $q(f)$ and $h_1\circ q(f)$ are isotopic by construction. It is easy to see that the left square in the diagram commutes by applying the definition of $\widehat p$ to deck transformations. 

Finally, regarding the claim that the diagram is a pullback, we show that the map to the fibered product 
\[\widehat p\times q:\Homeo(\widehat Y)^\Delta\to\Mod(Y,*)\times_{\Mod(Y)}\Homeo(Y)\]
is an isomorphism. The codomain consists of pairs $(\phi,g)\in\Mod(Y,*)\times\Homeo(Y)$ such that $g$ is isotopic to a representative of the isotopy class $\phi$. 

We define an inverse $\iota$ to $\widehat p\times q$. Given $(\phi,g)$ in the fibered product, choose an isotopy $g_t$ from $g$ to a homeomorphism representing $\phi$. Lift $g_t$ to an isotopy $\widetilde g_t$ such that $\widetilde g_1$ fixes $\widehat*$, and define $\iota(\phi,g)=\widetilde g_0$. The reader can check that the maps $\iota$ and $\widehat p\times q$ are inverses. 
\end{proof}

\subsection{Step 2: finite group action rigidity} 

Recall that $g=4k-1$ and $k\ge3$ is not a power of 2; let $p$ be an odd prime dividing $k$. From Step 1, we have homomorphism $\what\sigma:\Mod(\Si,*)\to\Homeo(\what X)$ that descends to a splitting $\sigma:\Mod(\Si)\to\Homeo(X)$. In this section we describe the action of a particular finite subgroup of $\Mod(\Si,*)$ on $\what X$. 

\begin{proposition}\label{prop:fixed-set}
There exists an element $\alpha\in\Mod(\Sigma,*)$ of order $2p$ such that the fixed sets of $\what\sigma(\alpha)$, $\what\sigma(\alpha)^2$, and $\what\sigma(\alpha)^p$ coincide and are equal to an embedded circle $c\subset\what X$. 
\end{proposition}

It is worth noting that the fixed set of a finite-order, orientation-preserving homeomorphism of a 3-manifold can be wildly embedded \cite{MZ}. 

In order to prove Proposition \ref{prop:fixed-set} we first construct the specific element $\alpha$. Then we prove (Proposition \ref{prop:fixed-set-smooth}) a weaker version of Proposition \ref{prop:fixed-set} with the additional assumption that the action is smooth. Finally, we combine this with a result of Pardon \cite{pardon} and Smith theory to prove Proposition \ref{prop:fixed-set}.

\p{\boldmath Construction of $\alpha$} We obtain $\alpha$ as an element in a dihedral subgroup $D_{4k}$ of $\Mod(\Sigma)$, where $D_{4k}$ denotes the dihedral group of order $8k$. The dihedral action $D_{4k}\car\Sigma$ we use has quotient $\Sigma/D_{4k}$ homeomorphic to $T^2$ and the quotient $\Sigma\to\Sigma/D_{4k}$ has a single branch point; it is determined by the homomorphism
\[\langle x,y\rangle = F_2\cong\pi_1(T^2\setminus\pt)\to D_{4k}=\langle a,b\mid a^{4k}=b^2=1, bab=a^{-1}\rangle\]
\[x\mapsto a,\>\>\>\> y\mapsto b.\]
By Riemann--Hurwitz, the genus of $\Sigma$ is $4k-1$. The orbifold $O=\Sigma/D_{4k}$ has fundamental group
\[
\pi_1^{orb}(O)=\langle x,y,h\mid h^{2k}=1,h=[x,y]\rangle,
\]
and there is a short exact sequence
\begin{equation}\label{eqn:orbifold-pi1}
1\to \pi_1(\Sigma)\to \pi_1^{orb}(O)\to D_{4k}\to 1.\end{equation}
This sequence induces a homomorphism $\pi_1^{orb}(O)\to\Mod(\Si,*)$. We take $\alpha=h^{k/p}$, where $p$, as defined above, is an odd prime dividing $k$, which exists by assumption. Then $\alpha$ is an element of order $2p$ in the subgroup $\langle h\rangle\cong\Z/2k\Z$ of $\pi_1^{orb}(O)<\Mod(\Si,*)$. 

\begin{remark}
The argument that follows works equally well when $\Sigma/D_{4k}$ is a genus-$g$ surface and $\Sigma\to\Sigma/D_{4k}$ has a single branched point. This provides more values of $g,e$ for which the conclusion of Theorem \ref{thm:main} holds. 
\end{remark}

\p{Smooth case} 
Here we prove the following proposition. 

\begin{proposition}\label{prop:fixed-set-smooth} 
Fix $D_{4k}<\Mod(\Sigma)$ as above. Suppose that 
$\sigma:D_{4k}\to\Diff^+(X)$ and is a splitting of $\pi:\Homeo^+(X)\to\Mod(\Si)$ over $D_{4k}$. Then $\what\sigma(\alpha)$ fixes a unique circle on $\what X=\mathbb H^2\times S^1$. Consequently, the fixed set of $\sigma(a^{2k/p})$ is nonempty. 
\end{proposition}

The last part of the statement of Proposition \ref{prop:fixed-set-smooth} follows from the preceding statement because the image of $\alpha$ under $\pi_1^{orb}(O)\to D_{4k}$ is $a^{2k/p}$.

\begin{proof}[Proof of Proposition \ref{prop:fixed-set-smooth}]
First we reduce to a more geometric setting. By Meeks--Scott \cite[Thm.\ 2.1]{meeks-scott}, the smooth(!) action $\sigma(D_{4k})\car X$ preserves some geometric metric on $X$. There are two possibilities for the geometry: if $e(X)=0$, then $X$ has $\mathbb{H}^2\times \mathbb{R}$-geometry, and if $e(X)\neq 0$, then $X$ has $\widetilde{\PSL_2(\mathbb{R})}$-geometry. We treat these cases in parallel. 

The universal cover $\widetilde X$ (with the induced geometric structure) is either $\mathbb H^2\times\R$ or $\widetilde{\PSL_2(\R)}$. In either case, $\wtil X$ has an isometric foliation by lines whose leaf space is isometric to $\mathbb H^2$, and this foliation is preserved by $\Isom(\wtil X)$, so there is a homomorphism $\Isom(\wtil X)\to\Isom(\mathbb H^2)$. Let $\Isom^+(\wtil X)<\Isom(\wtil X)$ be the group whose action on the leaves and on the leaf space are both orientation preserving. There is an exact sequence
\begin{equation}\label{eqn:isom-groups}
1\to\R\to \Isom^+(\wtil X)\xrightarrow{F}\Isom^+(\mathbb H^2)\to1.\end{equation}
See also \cite[\S4]{Scott}. 

Next consider the group $\Lambda$ of all lifts of elements of $\sigma(D_{4k})<\Isom(X)$ to $\Isom^+(\wtil X)$. This yields an exact sequence 
\[1\to\pi_1(X)\to\Lambda\to D_{4k}\to1.\]
The action of $\Lambda$ on $\wtil X$ induces an action of $\Lambda/\zeta$ on $\wtil X/\zeta=\what X\cong\mathbb H^2\times S^1$, where $\zeta$ is the center of $\pi_1(X)$. This action extends to an action of $\Isom^+(\wtil X)/\zeta$, and there is a homomorphism
\[\rho:\Lambda/\zeta\to\Isom^+(\wtil X)/\zeta\xrightarrow{\cong}\Isom(\what X).\]
The last map is an isomorphism by the general formula $\Isom(\wtil X/\Lambda)=N_{\Isom(\wtil X)}(\Lambda)/\Lambda$ for discrete subgroups $\Lambda<\Isom(\wtil X)$. 

To prove the proposition, we first identify $\Lambda/\zeta$ with $\pi_1^{orb}(O)$ (Claim \ref{claim:SES}). Then it is a formal consequence of our setup that $\rho(h^{k/p})=\what\sigma(\alpha)$, and after showing $\Isom^+(\wtil X)/\zeta\cong\Isom^+(\mathbb H^2)\times\SO(2)$ (Claim \ref{claim:isom-quotient}), we show that $\rho(h^{k/p})$ fixes a unique circle in $\what X$ (Claim \ref{claim:fixed-circle}). 

\begin{claim}\label{claim:SES}The restriction of the sequence (\ref{eqn:isom-groups}) to $\Lambda$ is a short exact sequence
\[1\to\zeta\to\Lambda\to\pi_1^{orb}(O)\to1\]
where $\zeta$ is the center of $\pi_1(X)$. 
\end{claim}

\begin{proof}[Proof of Claim \ref{claim:SES}]
Recall the map $F:\Isom^+(\wtil X)\to\Isom^+(\mathbb H^2)$ from (\ref{eqn:isom-groups}). First we identify $F(\Lambda)<\Isom^+(\mathbb H^2)$ with $\pi_1^{orb}(O)$. For this, it suffices to show that $F(\Lambda)$ fits into a short exact sequence 
\begin{equation}\label{eqn:orbifold-extension}1\to\pi_1(\Sigma)\to F(\Lambda)\to D_{4k}\to1,\end{equation}
where the ``monodromy" $D_{4k}\to\Out^+\big(\pi_1(\Sigma)\big)\cong\Mod(\Sigma)$ has image the given subgroup $D_{4k}<\Mod(\Si)$. This implies that $F(\Lambda)\cong\pi_1^{orb}(O)$ because $\pi_1^{orb}(O)$ is an extension of the same form (see (\ref{eqn:orbifold-pi1})), and extensions of $\pi_1(\Sigma)$ are determined by their monodromy \cite[\S IV.3]{Brown}. 

To construct the extension (\ref{eqn:orbifold-extension}), first note that the restriction of (\ref{eqn:isom-groups}) to $\pi_1(X)$ is the short exact sequence 
\[1\to\zeta\to\pi_1(X)\to\pi_1(\Sigma)\to1.\] 
The group $\pi_1(\Sigma)=F(\pi_1(X))$ is normal in $F(\Lambda)$ because $\pi_1(X)$ is normal in $\Lambda$. Furthermore, the surjection $\Lambda\to F(\Lambda)$ induces a surjection $D_{4k}=\Lambda/\pi_1(X)\to F(\Lambda)/\pi_1(\Sigma)$. 

The quotient map $\wtil X\to\mathbb H^2$, which is equivariant with respect to $\Lambda\to F(\Lambda)$ descends to a map 
$X=\wtil X/\pi_1(X)\to \mathbb H^2/\pi_1(\Sigma)=\Sigma$ that's equivariant with respect to $D_{4k}=\Lambda/\pi_1(X)\to F(\Lambda)/\pi_1(\Sigma)$. 

Since $\sigma$ is a realization, the induced action of $\sigma(D_{4k})$ on $\Sigma$ is a realization of the $D_{4k}<\Mod(\Sigma)$, and in particular the $D_{4k}$ action on $\Sigma$ is faithful. Therefore, $F(\Lambda)/\pi_1(\Sigma)\cong D_{4k}$, and the monodromy of the associated extension
\[1\to\pi_1(\Sigma)\to F(\Lambda)\to D_{4k}\to1,\]
is the given inclusion $D_{4k}<\Mod(\Sigma)$. This concludes the proof that $F(\Lambda)$ is isomorphic to $\pi_1^{orb}(O)$. 

To finish the proof of Claim \ref{claim:SES}, it remains to show that the intersection of $\Lambda$ with $\R=\ker(F)$ is $\zeta$. We do this by showing (i) $\Lambda\cap\R$ is the center of $\Lambda$, and (ii) the center of $\Lambda$ is contained in $\pi_1(X)$. Together with the obvious containment $\zeta<\Lambda\cap\R$, (i) and (ii) imply $\Lambda\cap\R=\zeta$. 

(i): First note that $\Lambda\cap\R$ is central because $\mathbb R$ is central in $\Isom(\wtil X)$. On the other hand, the center of $\Lambda$ is contained in $\Lambda\cap\R$ because the center of $\Lambda/(\Lambda\cap\R)\cong\pi_1^{orb}(O)$ has trivial center. 

(ii): To show the center of $\Lambda$ is contained in $\pi_1(X)$, we show that the center of $\Lambda$ projects trivially to $D_{4k}=\Lambda/\pi_1(X)$. This is true because $\Lambda\to D_{4k}$ factors through $\pi_1^{orb}(O)$, which has trivial center.\end{proof}

We summarize the relation between the relevant groups in Diagram (\ref{eqn:extensions}). 

\begin{equation}\label{eqn:extensions}
\begin{xy}
(-40,0)*+{1}="A";
(-20,0)*+{\pi_1(X)}="B";
(0,0)*+{\Lambda}="C";
(20,0)*+{D_{4k}}="D";
(40,0)*+{1}="E";
(-40,-15)*+{1}="F";
(-20,-15)*+{\pi_1(\Sigma)}="G";
(0,-15)*+{\pi_1^{orb}(O)}="H";
(20,-15)*+{D_{4k}}="I";
(40,-15)*+{1}="J";
(-20,15)*+{\zeta}="X";
(0,15)*+{\zeta}="Y";
(-20,25)*+{1}="Z";
(0,25)*+{1}="W";
(-20,-25)*+{1}="U";
(0,-25)*+{1}="V";
{\ar"A";"B"}?*!/_3mm/{};
{\ar "B";"C"}?*!/_3mm/{};
{\ar "C";"D"}?*!/_3mm/{};
{\ar "D";"E"}?*!/_3mm/{};
{\ar"F";"G"}?*!/_3mm/{};
{\ar "G";"H"}?*!/_3mm/{};
{\ar "H";"I"}?*!/_3mm/{};
{\ar "I";"J"}?*!/_3mm/{};
{\ar "B";"G"}?*!/_3mm/{};
{\ar "C";"H"}?*!/_3mm/{};
{\ar@{=} "D";"I"}?*!/_3mm/{};
{\ar "X";"B"}?*!/_3mm/{};
{\ar "Y";"C"}?*!/_3mm/{};
{\ar "Z";"X"}?*!/_3mm/{};
{\ar "W";"Y"}?*!/_3mm/{};
{\ar "G";"U"}?*!/_3mm/{};
{\ar "H";"V"}?*!/_3mm/{};
{\ar@{=} "X";"Y"}?*!/_3mm/{};
\end{xy}\end{equation}

By Claim \ref{claim:SES}, $\Lambda/\zeta\cong\pi_1^{orb}(O)$, so $\rho$ takes the form
\[\rho:\pi_1^{orb}(O)\to\Isom^+(\wtil X)/\zeta\cong\Isom(\what X)\]
By construction, this homomorphism is the restriction of $\what\sigma:\Mod(\Si,*)\to\Homeo(\what X)$ to $\pi_1^{orb}(O)$. 
Since $\alpha=h^{k/p}$, to show the fixed set of $\what\sigma(\alpha)$ is a circle, it suffices to show the same statement for $\rho(h^{k/p})$. To prove this, we first compute $\Isom(\what X)\cong\Isom^+(\wtil X)/\zeta$. 

\begin{claim}\label{claim:isom-quotient}
The group $\Isom^+(\wtil X)/\zeta$ is isomorphic to $\Isom^+(\mathbb H^2)\times \SO(2)$. 
\end{claim}

\begin{proof}[Proof of Claim \ref{claim:isom-quotient}]
First note that there is an extension
\[1\to \SO(2)\to \Isom^+(\wtil X)/\zeta\to \Isom^+(\mathbb H^2)\to1\]
induced from (\ref{eqn:isom-groups}). This sequence is obviously split when $\wtil X=\mathbb H^2\times\R$ since $\Isom^+(\wtil X)\cong\Isom^+(\mathbb H^2)\times\R$ is a product. 

Assume now that $\wtil X=\wtil{\PSL_2(\R)}$, and write $e=(2g-2)n$ where $n$ is a nonzero integer. Let $K$ denote the kernel of the universal cover homomorphism $\wtil{\PSL_2(\R)}\to\PSL_2(\R)$. 

We claim that $\zeta=\frac{1}{n}K$. To see this, note that the extension 
\[1\to K\to\wtil{\PSL_2(\R)}\to\PSL_2(\R)\to1\] pulled back under a Fuchsian representation $\pi_1(\Si)\to\PSL_2(\R)$ induces the extension of the unit tangent bundle group $\pi_1(U\Si)$, which has Euler number $2-2g$, and there is an $n$-fold fiberwise cover $U\Si\to X$, so the center of $\pi_1(U\Sigma)<\pi_1(X)$ is generated by the $n$-the power of the generator of the center of $\pi_1(X)$, i.e.\ $\zeta=\frac{1}{n}K$.


The inclusion of $\wtil{\PSL_2(\R)}$ in $\Isom\big(\wtil{\PSL_2(\R)}\big)$ (given by left-multiplication) descends to a homomorphism 
\[\PSL_2(\R)=\wtil{\PSL_2(\R)}/K\to\Isom\big(\wtil{\PSL_2(\mathbb R)}\big)/K\onto\Isom\big(\wtil{\PSL_2(\mathbb R)}\big)/\zeta\] that defines a splitting of the sequence \[1\to\SO(2)\to\Isom(\wtil{\PSL_2(\R)})/\zeta\to\PSL_2(\R)\to1.\qedhere\]
\end{proof}

The following Claim \ref{claim:fixed-circle} is the last step in the proof of Proposition \ref{prop:fixed-set-smooth}. 

\begin{claim}\label{claim:fixed-circle}
Let $p$ be an odd prime dividing $k$. If $e$ is divisible by $(2g-2)2p$, then the fixed set of $\rho(h^{k/p})$ is a circle. 
\end{claim}

Before proving the claim, we explain how the factors of $\Isom^+(\mathbb H^2)\times\SO(2)\cong \Isom^+(\what X)$ act on $\what X=\wtil X/\zeta$. 

\begin{remark}\label{rmk:isomH-action}
Consider the isomorphism $\Isom(\what X)\cong\Isom^+(\mathbb H^2)\times\SO(2)$ from Claim \ref{claim:isom-quotient}. In each case ($e=0$ or $e\neq0$) the action of $\SO(2)$ on $\what X$ covers the identity of $\mathbb H^2$ and acts freely by rotation on the circle fibers of $X\to\mathbb H^2$. For the $\Isom^+(\mathbb H^2)$ action, when $e=0$, then $\what X\cong\mathbb H^2\times S^1$ is a metric product, and the action of $\Isom^+(\mathbb H^2)$ is trivial on the $S^1$ factor and is the natural action on $\mathbb H^2$. If $e=(2g-2)n$ is nonzero, then 
\[\what X\cong\wtil{\PSL_2(\R)}/\zeta\cong\PSL_2(\R)/(\Z/n\Z),\]
and with respect to this isomorphism, the action of $\Isom^+(\mathbb H)\cong\PSL_2(\R)$ on $\what X$ is induced from left multiplication of $\PSL_2(\R)$ on $\PSL_2(\R)$. Identifying $\PSL_2(\R)$ with the unit tangent $U\mathbb H^2$, we can also view $\PSL_2(\R)/(\Z/n\Z)$ as the quotient of $U\mathbb H^2$ by the $\Z/n\Z$ action that covers the identity of $\mathbb H^2$ and rotates each fiber. 
\end{remark}

\begin{proof}[Proof of Claim \ref{claim:fixed-circle}]
Write $e=(2g-2)2pm$ for some integer $m$. 

First note that since $\rho(h)$ has finite order, the induced isometry of $\mathbb H^2$ has a unique fixed point, so $\rho(h)$ preserves a unique circle $C$ of the fibering $\what X\to\mathbb H^2$. The same is true for $\rho(h^{k/p})$, and we will show that $\rho(h^{k/p})$ acts trivially on $C$. 

Since $h=[x,y]$ is a commutator in $\pi_1^{orb}(O)$ and $\SO(2)$ is abelian, we find that the projection 
\[\rho(h)\in\Isom^+(\what X)\cong\Isom^+(\mathbb H^2)\times\SO(2)\to\SO(2)\]
is trivial. Therefore, the action of $\rho(h)$ on $\what X$ factors through $\Isom^+(\mathbb H^2)$ acting on $\what X$. This action is described in Remark \ref{rmk:isomH-action}. If $e=0$, since $\Isom^+(\mathbb H^2)$ acts trivially on the $S^1$ factor of $\what X\cong\mathbb H^2\times S^1$, we conclude that $\rho(h)$ acts trivially on $C$. If $e\neq0$, then $\rho(h)$ acts as a a rotation by $2\pi(pm/k)$ on $C$, so $\rho(h^{k/p})$ acts as a rotation by $2\pi m$, which is trivial. 
\end{proof}

This completes the proof of Proposition \ref{prop:fixed-set-smooth}.
\end{proof}

\p{Homeomorphism case} Here prove Proposition \ref{prop:fixed-set}. 

\begin{proof}[Proof of Proposition \ref{prop:fixed-set}]
By Pardon \cite[Thm.\ 1.1]{pardon}, there is a sequence of smooth $D_{4k}$ actions converging in $\Hom\big(D_{4k},\Homeo(X)\big)$ to the given action of $\sigma(D_{4k})$ on $X$. Sufficiently close approximates also give a splitting of $\pi$ over $D_{4k}<\Mod(\Sigma)$ because $\Homeo(X)$ is locally path connected \cite{Kirby-Edwards}. 

For each of the smooth approximations of $\sigma(D_{4k})$, the fixed set of $a^{2k/p}$ is nonempty by Proposition \ref{prop:fixed-set-smooth}. This implies that $\sigma(a^{2k/p})$ has a fixed point (a sequence of fixed points, one for each smooth action, sub-converges to a fixed point of the $\sigma(a^{2k/p})$ action). By Remark \ref{rmk:lift-fp}, there exists a lift of $a^{2k/p}\in D_{4k}<\Mod(\Sigma)$ to a finite order element $\alpha'\in \pi_1^{orb}(O)<\Mod(\Sigma,*)$ so that $\what\sigma(\alpha')$ has a fixed point. Since $\pi_1^{orb}(O)$ has a unique conjugacy class of finite subgroup of order $2p$, the subgroups $\langle\alpha'\rangle$ and $\langle\alpha\rangle$ are conjugate, so the fixed set of $\what\sigma(\alpha)$ is nonempty. 


It remains to show the fixed set of $\what\sigma(\alpha)$ is a circle, and that this circle is the same as the fixed sets of $\what\sigma(\alpha)^2$ and $\what\sigma(\alpha)^p$. 

First we show (using Smith theory) that both $\what\sigma(\alpha)^2$ and $\what\sigma(\alpha)^p$ have fixed set a single circle (we are not yet claiming/arguing that the fixed sets of $\what\sigma(\alpha)^2$ and $\what\sigma(\alpha)^p$ are the same). To see this, we focus on $\what\sigma(\alpha)^2$ for concreteness. Consider the group $\Lambda_0$ of all lifts of powers of $\what\sigma(\alpha)^2$ to the universal cover $\widetilde X$. This group is an extension 
\[1\to\Z\to\Lambda_0\to\Z/p\Z\to1,\]
which is central and split; hence $\Lambda_0\cong\Z\times\Z/p\Z$. It is central because $\alpha$ acts orientation-preservingly on fibers of $X\to\Sigma$ (otherwise, the action of $\alpha$ on $\Sigma$ would reverse orientation, contrary to the construction); it splits because $\what\sigma(\alpha)$ has a fixed point. The (lifted) action of $\what\sigma(\alpha)^2$ on $\widetilde X$ has fixed set a line (i.e.\ embedded copy of $\R$) by Smith theory and local Smith theory \cite[Theorem 20.1]{Bredon}, and this line is preserved and acted properly by $\Z<\Lambda_0$; thus $\what\sigma(\alpha)^2$ acts on $\what X$ with a circle in its fixed set. Furthermore, each component of the fixed set of $\what\sigma(\alpha)^2$ acting on $\what X$ corresponds to a distinct conjugacy class of order-$p$ subgroup of $\Lambda_0$. Since there is only one $\Z/p\Z$ subgroup of $\Lambda_0$, the fixed set of $\what\sigma(\alpha)^2$ is connected, i.e.\ a single circle. The same argument\footnote{Smith theory applies to prime-order finite cyclic group actions, so we cannot apply this argument directly to $\what\sigma(\alpha)$.} works for $\what\sigma(\alpha)^p$. 

Now we determine the fixed set of $\what\sigma(\alpha)$. First observe that $\what\sigma(\alpha)$ preserves the fixed set of $\what\sigma(\alpha)^2$ and has a fixed point there (the fixed set of $\what\sigma(\alpha)$ is nonempty and contained in the fixed set of $\what\sigma(\alpha)^2$). The only $\Z/2p\Z$ action on the circle with a fixed point is the trivial action, so in fact the fixed sets of $\what\sigma(\alpha)$ and $\what\sigma(\alpha)^2$ are the same. The same argument applies to $\what\sigma(\alpha)$ and $\what\sigma(\alpha)^p$. This proves Proposition \ref{prop:fixed-set}. 
\end{proof}

\subsection{Step 3: centralizer argument}

Recall that we have defined $\alpha$ as an element of order $2p$ in $\pi_1^{orb}(O)<\Mod(\Sigma,*)$. In this step we prove that $\Mod(\Sigma,*)$ is generated by the centralizers of $\alpha^2$ and $\alpha^p$.

\begin{proposition}[centralizer property]\label{prop:centralizer}
Let $\alpha\in\Mod(\Sigma,*)$ be the element of order $2p$ constructed above. Then 
\[\Mod(\Sigma,*)=\langle C(\alpha^2),C(\alpha^p)\rangle,\] where $C(-)$ denotes the centralizer in $\Mod(\Sigma,*)$. 
\end{proposition}

\p{Strategy for proving Proposition \ref{prop:centralizer}} 
Set $\Gamma=\langle C(\alpha^2),C(\alpha^p)\rangle$. Our method for showing $\Gamma=\Mod(\Sigma,*)$, which is similar to the proof of \cite[Thm.\ 1.1]{Chen-Salter}, is to inductively build subsurfaces 
\begin{equation}\label{eqn:subsurfaces}S_0\subset S_1 \subset \cdots \subset S_N \subset \Sigma\setminus\{*\}\end{equation} such that $\Mod(S_n)\subset \Gamma$ for each $n$ and $S_N$ fills $\Sigma\setminus\{*\}$ (i.e.\ each boundary component of $S_N$ is inessential in $\Sigma\setminus\{*\}$). The fact that $S_N$ fills implies that $\Mod(S_N)=\Mod(\Sigma,*)$, so then $\Mod(\Sigma,*)\subset\Gamma$ by the last step in the inductive argument. 

In order to ensure that $\Mod(S_n)\subset\Gamma$, the subsurface $S_n$ is obtained from $S_{n-1}$ by an operation known as \emph{subsurface stabilization}. If  $S\subset \Sigma$ is a subsurface and $c\subset \Sigma$ is a simple closed curve that intersects $S$ in a single arc, then the \emph{stabilization of $S$ along $c$} is the subsurface $S\cup N(c)$, where $N(c)$ is a regular neighborhood of $c$. It is easy to show that $\Mod(S\cup N(c))$ is generated by $\Mod(S)$ and the Dehn twist $\tau_c$ \cite[Lem.\ 4.2]{Chen-Salter}, so if $\Mod(S)\subset\Gamma$ and $\tau_c\in\Gamma$, then $\Mod(S\cup N(c))\subset\Gamma$. Therefore, for the proof, it suffices to find a sequence of subsurface stabilizations along curves whose Dehn twist belongs to $\Gamma=\langle C(\alpha^2),C(\alpha^p)\rangle$.

\p{\boldmath Model for the $\alpha$ action} Our proof of Proposition \ref{prop:centralizer} makes use of an explicit model for $\Sigma$ with its $\alpha$ action, which is pictured below in the case $k=6$ and $p=3$ (recall that $g=4k-1$ and $p$ is an odd prime dividing $k$).  

\vspace{.1in}
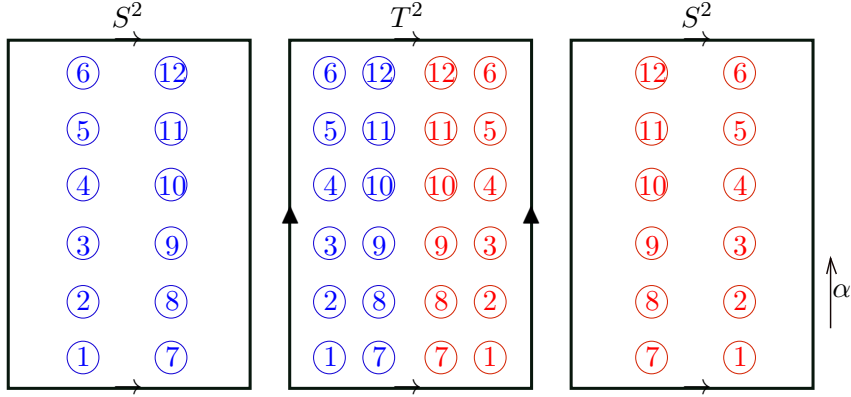
\begin{figure}[h!]
\labellist
\hair 2pt
\pinlabel \textcolor{blue}{$1$} at 87 154
\pinlabel \textcolor{blue}{$2$} at 87 208
\pinlabel \textcolor{blue}{$3$} at 87 262
\pinlabel \textcolor{blue}{$4$} at 87 320
\pinlabel \textcolor{blue}{$5$} at 87 373
\pinlabel \textcolor{blue}{$6$} at 87 429

\pinlabel \textcolor{blue}{$7$} at 173 154
\pinlabel \textcolor{blue}{$8$} at 173 208
\pinlabel \textcolor{blue}{$9$} at 173 262
\pinlabel \textcolor{blue}{$10$} at 173 320
\pinlabel \textcolor{blue}{$11$} at 173 373
\pinlabel \textcolor{blue}{$12$} at 173 429

\pinlabel \textcolor{blue}{$1$} at 326 154
\pinlabel \textcolor{blue}{$2$} at 326 208
\pinlabel \textcolor{blue}{$3$} at 326 262
\pinlabel \textcolor{blue}{$4$} at 326 320
\pinlabel \textcolor{blue}{$5$} at 326 373
\pinlabel \textcolor{blue}{$6$} at 326 429

\pinlabel \textcolor{blue}{$7$} at 373 154
\pinlabel \textcolor{blue}{$8$} at 373 208
\pinlabel \textcolor{blue}{$9$} at 373 262
\pinlabel \textcolor{blue}{$10$} at 373 320
\pinlabel \textcolor{blue}{$11$} at 373 373
\pinlabel \textcolor{blue}{$12$} at 373 429

\pinlabel \textcolor{red}{$7$} at 433 154
\pinlabel \textcolor{red}{$8$} at 433 208
\pinlabel \textcolor{red}{$9$} at 433 262
\pinlabel \textcolor{red}{$10$} at 433 320
\pinlabel \textcolor{red}{$11$} at 433 373
\pinlabel \textcolor{red}{$12$} at 433 429

\pinlabel \textcolor{red}{$1$} at 480 154
\pinlabel \textcolor{red}{$2$} at 480 208
\pinlabel \textcolor{red}{$3$} at 480 262
\pinlabel \textcolor{red}{$4$} at 480 320
\pinlabel \textcolor{red}{$5$} at 480 373
\pinlabel \textcolor{red}{$6$} at 480 429

\pinlabel \textcolor{red}{$7$} at 637 154
\pinlabel \textcolor{red}{$8$} at 637 208
\pinlabel \textcolor{red}{$9$} at 637 262
\pinlabel \textcolor{red}{$10$} at 637 320
\pinlabel \textcolor{red}{$11$} at 637 373
\pinlabel \textcolor{red}{$12$} at 637 429

\pinlabel \textcolor{red}{$1$} at 722 154
\pinlabel \textcolor{red}{$2$} at 722 208
\pinlabel \textcolor{red}{$3$} at 722 262
\pinlabel \textcolor{red}{$4$} at 722 320
\pinlabel \textcolor{red}{$5$} at 722 373
\pinlabel \textcolor{red}{$6$} at 722 429

\pinlabel {$\alpha$} at 820 218
\pinlabel $T^2$ at 400 485
\pinlabel $\ra$ at 400 460
\pinlabel $\ra$ at 400 123
\pinlabel $\blacktriangle$ at 287 291
\pinlabel $\blacktriangle$ at 520 291
\pinlabel $S^2$ at 130 485
\pinlabel $\ra$ at 130 460
\pinlabel $\ra$ at 130 123
\pinlabel $S^2$ at 680 485
\pinlabel $\ra$ at 680 460
\pinlabel $\ra$ at 680 123
\endlabellist
\centering
\includegraphics[scale=.4]{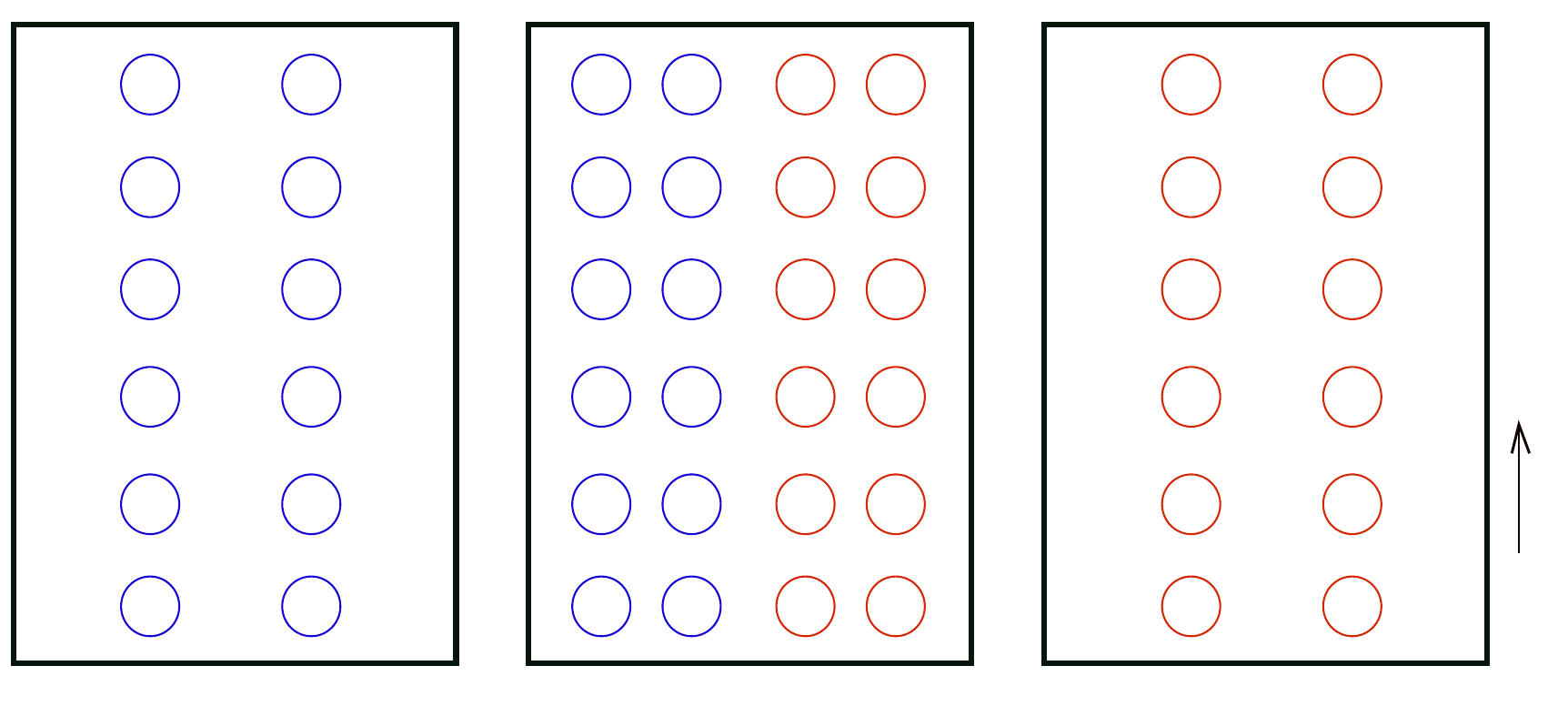}
\caption{The model of the surface $\Sigma_{23}$ where $p=3$. } 
\label{fig:model} \end{figure}

The surface $\Sigma$ is built out of two copies of the standard action of $\Z/2p\Z$ on $S^2$ and one copy of a free action of $\Z/2p\Z$ on $T^2$. We glue each copy of $S^2$ to $T^2$ along $k/p$ free orbits by an equivariant connected sum. In the figure, $\alpha$ acts by vertical translation. Note that the fixed points of $\alpha$ on $S^2$ are not pictured in the figure -- they are at $\pm\infty$ along the $x$-axis. 

To derive this model, recall that $D_{4k}=\langle a,b\mid a^{4k}=1=b^2, bab=a^{-1}\rangle$ has abelianization $D_{4k}\to(\Z/2\Z)^2$ with kernel $\langle a^2\rangle\cong\Z/2k\Z$. Then there is a sequence of regular covers 
\[\Sigma\xrightarrow{\langle a^2\rangle}\Sigma/\langle a^2\rangle\xrightarrow{(\Z/2\Z)^2}\Sigma/D_{4k}.\]
The cover $\Sigma/\langle a^2\rangle\to\Sigma/D_{4k}$ is unbranched and is the $\Z/2\Z$-homology cover of $T^2$ (in particular, $\Sigma/\langle a^2\rangle$ is also a torus). The cover $\Sigma\to\Sigma/\langle a^2\rangle$ is branched over four points;  the local monodromy around the branched points is $a^2$ at two of the branched points and $a^{-2}$ at the other two. Choosing branched cuts joining branched points in $a^{\pm2}$ pairs gives a model for $\Sigma$, and one can check that this model is equivalent to the one described above. (The spheres in Figure \ref{fig:model} arise from pre-images under $\Sigma\to\Sigma/\langle a^2\rangle$ of neighborhoods of the branch cuts.)  


By Remark \ref{rmk:lift-fp}, since $\sigma(a^{2k/p})$ has a fixed point, the subgroup $\langle\alpha\rangle\subset \Mod(\Sigma,*)$, which lifts $\langle a^{2k/p}\rangle$, has a fixed point. The different lifts of $\langle a^{2k/p}\rangle$ to a finite subgroup of $\Mod(\Sigma,*)$ are in one-to-one correspondence to fixed points of $a^{2k/p}$. Since these fixed points are permuted transitively by the action of $D_{4k}$, the different lifts of $\langle a^{2k/p}\rangle$ are conjugate. Consequently, for the purpose of our argument, we can choose $*$ to be any one of the four fixed points of $\alpha$ and prove Proposition \ref{prop:centralizer} for this choice, without loss of generality. (It will also be evident from the argument that a similar argument applies if $*$ is changed to another fixed point.) 




\begin{remark}We do not known how generally the relation $\Mod(\Sigma,*)=\langle C(\alpha^2),C(\alpha^p)\rangle$ holds. For example, it may hold for every $\Z/2p\Z$ subgroup of $\Mod(\Sigma,*)$. We do not know a general (abstract) approach to this problem.
\end{remark} 

\begin{proof}[Proof of Proposition \ref{prop:centralizer}]\ 

\p{Symmetry breaking} In preparation for constructing a sequence of subsurface stabilizations, in this paragraph we find a suitable collection of Dehn twists that belong to $\Gamma$. 
The obvious way for $\tau_c$ to belong to $\Gamma$ is if $c$ is preserved by either $\alpha^2$ or $\alpha^p$. More generally, we use a process that we call \emph{symmetry breaking} to show $\tau_c\in\Gamma$ for certain $c$. We formulate this in the following lemma, which is similar to \cite[Lem.\ 3.2]{Chen-Salter}. 

\begin{lemma}[Symmetry breaking] \label{lem:symmetry}
Assume that $c,d\subset \Sigma$ are simple closed curves that intersect once and $\tau_d\in\Gamma$. Suppose that  either \emph{(i)} $\alpha^p(c)$ is disjoint from $c$ and $d$ or \emph{(ii)} the curves $\alpha^2(c),\alpha^4(c)\ldots,\alpha^{2p-2}(c)$ are disjoint from $d$ and the curves $c,\alpha^2(c),\alpha^4(c)\ldots,\alpha^{2p-2}(c)$ are pairwise disjoint. Then $\tau_c\in\Gamma$. 
\end{lemma}


\begin{proof}[Proof of Lemma \ref{lem:symmetry}]
We prove case (i) of the statement; case (ii) is similar. Since Dehn twists about disjoint curves commute, 
\[\big(\tau_c\tau_{\alpha^p(c)}\big)\tau_d\big(\tau_c\tau_{\alpha^p(c)}\big)^{-1} = \tau_c\tau_d\tau_c^{-1}.\]
The left hand side of the equation is in $\Gamma$ because $\tau_d\in\Gamma$ by assumption and $\tau_c\tau_{\alpha^p(c)}\in\Gamma$ because the curves $c,\alpha^p(c)$ are permuted by $\alpha^p$ and are disjoint (so their twists commute), and thus $\tau_c\tau_{\alpha^p(c)}\in C(\alpha^p)\subset \Gamma$. 
Since $c$ and $d$ intersect once, the braid relation implies that $\tau_d\tau_c\tau_d^{-1}= \tau_c\tau_d\tau_c^{-1}$ also belongs to $\Gamma$. Since $\tau_d\in \Gamma$ this implies that $\tau_c\in \Gamma$, as desired. 
\end{proof}

\begin{remark}
When applying Lemma \ref{lem:symmetry}(i) or (ii) we refer to it as the $\alpha^p$- or $\alpha^2$-symmetry breaking, respectively. 
\end{remark}

\begin{lemma}\label{lem:symmetry-breaking}
Dehn twists about the curves in Figure \ref{fig:symmetrybreaking} are in $\Gamma$. 
\end{lemma}

\vspace{.1in}
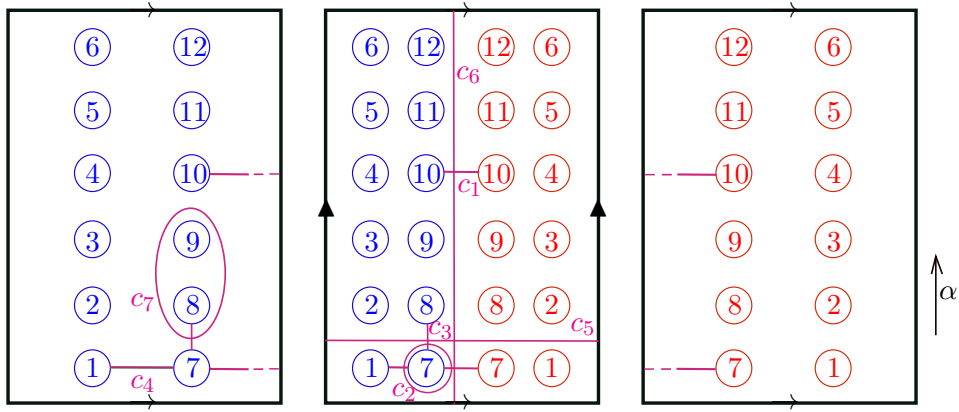
\begin{figure}[ht!]
\labellist
\hair 2pt
\pinlabel \textcolor{blue}{$1$} at 87 154
\pinlabel \textcolor{blue}{$2$} at 87 208
\pinlabel \textcolor{blue}{$3$} at 87 262
\pinlabel \textcolor{blue}{$4$} at 87 320
\pinlabel \textcolor{blue}{$5$} at 87 373
\pinlabel \textcolor{blue}{$6$} at 87 429

\pinlabel \textcolor{blue}{$7$} at 173 154
\pinlabel \textcolor{blue}{$8$} at 173 208
\pinlabel \textcolor{blue}{$9$} at 173 262
\pinlabel \textcolor{blue}{$10$} at 173 320
\pinlabel \textcolor{blue}{$11$} at 173 373
\pinlabel \textcolor{blue}{$12$} at 173 429

\pinlabel \textcolor{blue}{$1$} at 326 154
\pinlabel \textcolor{blue}{$2$} at 326 208
\pinlabel \textcolor{blue}{$3$} at 326 262
\pinlabel \textcolor{blue}{$4$} at 326 320
\pinlabel \textcolor{blue}{$5$} at 326 373
\pinlabel \textcolor{blue}{$6$} at 326 429

\pinlabel \textcolor{blue}{$7$} at 373 154
\pinlabel \textcolor{blue}{$8$} at 373 208
\pinlabel \textcolor{blue}{$9$} at 373 262
\pinlabel \textcolor{blue}{$10$} at 373 320
\pinlabel \textcolor{blue}{$11$} at 373 373
\pinlabel \textcolor{blue}{$12$} at 373 429


\pinlabel \textcolor{red}{$7$} at 433 154
\pinlabel \textcolor{red}{$8$} at 433 208
\pinlabel \textcolor{red}{$9$} at 433 262
\pinlabel \textcolor{red}{$10$} at 433 320
\pinlabel \textcolor{red}{$11$} at 433 373
\pinlabel \textcolor{red}{$12$} at 433 429

\pinlabel \textcolor{red}{$1$} at 480 154
\pinlabel \textcolor{red}{$2$} at 480 208
\pinlabel \textcolor{red}{$3$} at 480 262
\pinlabel \textcolor{red}{$4$} at 480 320
\pinlabel \textcolor{red}{$5$} at 480 373
\pinlabel \textcolor{red}{$6$} at 480 429

\pinlabel \textcolor{red}{$7$} at 637 154
\pinlabel \textcolor{red}{$8$} at 637 208
\pinlabel \textcolor{red}{$9$} at 637 262
\pinlabel \textcolor{red}{$10$} at 637 320
\pinlabel \textcolor{red}{$11$} at 637 373
\pinlabel \textcolor{red}{$12$} at 637 429

\pinlabel \textcolor{red}{$1$} at 722 154
\pinlabel \textcolor{red}{$2$} at 722 208
\pinlabel \textcolor{red}{$3$} at 722 262
\pinlabel \textcolor{red}{$4$} at 722 320
\pinlabel \textcolor{red}{$5$} at 722 373
\pinlabel \textcolor{red}{$6$} at 722 429

\pinlabel {$\alpha$} at 820 218

\pinlabel $\ra$ at 400 459
\pinlabel $\ra$ at 400 123
\pinlabel $\blacktriangle$ at 287 291
\pinlabel $\blacktriangle$ at 520 291

\pinlabel $\ra$ at 130 459
\pinlabel $\ra$ at 130 123

\pinlabel $\ra$ at 680 459
\pinlabel $\ra$ at 680 123

\pinlabel \textcolor{magenta}{$c_1$} at 410 308
\pinlabel \textcolor{magenta}{$c_2$} at 354 133
\pinlabel \textcolor{magenta}{$c_3$} at 385 186
\pinlabel \textcolor{magenta}{$c_4$} at 130 140
\pinlabel \textcolor{magenta}{$c_5$} at 507 189
\pinlabel \textcolor{magenta}{$c_6$} at 409 405
\pinlabel \textcolor{magenta}{$c_7$} at 130 210

\endlabellist
\centering
\includegraphics[scale=.45]{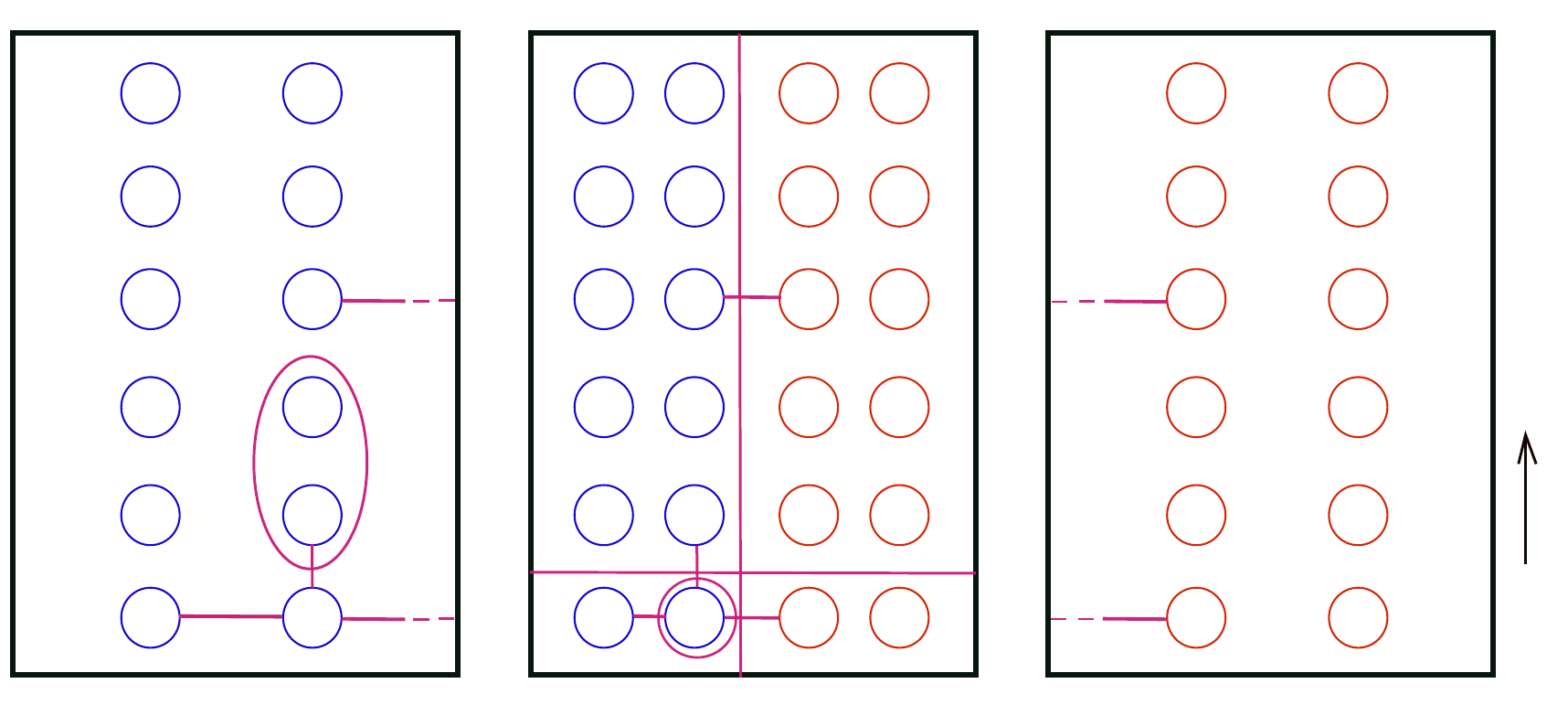}
\caption{The curves used in the proof of Lemma \ref{lem:symmetry-breaking}. Here we are using the model for $\Sigma_{23}$, but the proof follows in the same way for similar types of curves on any $\Sigma$.} 
\label{fig:symmetrybreaking} \end{figure}

In Figure \ref{fig:symmetrybreaking}, we illustrate the case $k=6,p=3$. The corresponding curves in the general case belong to $\Gamma$ by the exact same argument.

\begin{proof}[Proof of Lemma \ref{lem:symmetry-breaking}]
First observe that $\tau_{c_1}$ and $\tau_{c_6}$ are in $\Gamma$ because each is invariant under $\alpha^p$. 
We deduce that $\tau_{c_2}\in\Gamma$ using $\alpha^2$-symmetry breaking with $d=c_1$. 
Each of $\tau_{c_3}$ and $\tau_{c_4}$ are in $\Gamma$ by $\alpha^p$-symmetry breaking with $d=c_2$. Finally, both $\tau_{c_5}$ and $\tau_{c_7}$ are in $\Gamma$ by $\alpha^p$-symmetry breaking with $d=c_3$. 
\end{proof}

\p{Surface stabilization sequence} We stabilize with the sequence of curves represented in Figure \ref{fig:stab-sequence}.  To get the initial subsurface $S_0$ we can take the subsurface spanned by the chain of curves $c_0,c_1,\ldots,c_{4}$. This subsurface has genus 2 and one boundary component, and these curves are Humphries generators for $S_0$ \cite[Fig.\ 4.10]{Farb-Margalit}. 
Next we extend this chain with the curves $c_5,\ldots,c_{4k-1}$; at this point the left genus-0 subsurface has been filled. Next we stabilize with $c_{4k}$ and the curves ($c_{4k+1}, c_{4k+2}, \cdots, c_{8k-2}$) that fill the right genus-0 subsurface; there is some choice in the order of curves we stabilize, but this is not important. Finally we stabilize with the curves $c_{8k-1}$ and $c_{8k}$ that generate $\pi_1(T^2)$. 

All of the twists about the curves used are in $\Gamma$. In each case, this can be seen either directly from the statement of Lemma \ref{lem:symmetry-breaking} or by an argument that is a small variation of its proof. Since this collection of curves fills $\Sigma$, we have shown that $\Gamma=\Mod(\Sigma,*)$. This proves Proposition \ref{prop:centralizer}. 
\end{proof}

\vspace{.1in}
\begin{figure}[ht!]
\labellist
\hair 2pt
\pinlabel {$\alpha$} at 820 218

\pinlabel $\ra$ at 400 460
\pinlabel $\ra$ at 400 123
\pinlabel $\blacktriangle$ at 287 291
\pinlabel $\blacktriangle$ at 520 291

\pinlabel $\ra$ at 130 460
\pinlabel $\ra$ at 130 123

\pinlabel $\ra$ at 680 460
\pinlabel $\ra$ at 680 123


\hair 2pt

\pinlabel \textcolor{magenta}{$c_0$} at 45 430
\pinlabel \textcolor{blue}{$c_1$} at 88 430
\pinlabel \textcolor{magenta}{$c_2$} at 100 403
\pinlabel \textcolor{blue}{$c_3$} at 88 375
\pinlabel \textcolor{magenta}{$c_4$} at 100 340
\pinlabel \textcolor{blue}{$c_5$} at 88 320
\pinlabel \textcolor{magenta}{$c_6$} at 100 293
\pinlabel \textcolor{blue}{$c_7$} at 88 265
\pinlabel \textcolor{magenta}{$c_8$} at 100 235
\pinlabel \textcolor{blue}{$c_9$} at 88 205
\pinlabel \textcolor{magenta}{$c_{10}$} at 102 180
\pinlabel \textcolor{blue}{$c_{11}$} at 88 155
\pinlabel \textcolor{magenta}{$c_{12}$} at 130 141
\pinlabel \textcolor{blue}{$c_{13}$} at 173 155
\pinlabel \textcolor{magenta}{$c_{14}$} at 188 180
\pinlabel \textcolor{blue}{$c_{15}$} at 173 205
\pinlabel \textcolor{magenta}{$c_{16}$} at 188 235
\pinlabel \textcolor{blue}{$c_{17}$} at 173 265
\pinlabel \textcolor{magenta}{$c_{18}$} at 188 293
\pinlabel \textcolor{blue}{$c_{19}$} at 173 320
\pinlabel \textcolor{magenta}{$c_{20}$} at 188 347
\pinlabel \textcolor{blue}{$c_{21}$} at 173 375
\pinlabel \textcolor{magenta}{$c_{22}$} at 188 400
\pinlabel \textcolor{blue}{$c_{23}$} at 173 430

\pinlabel \textcolor{magenta}{$c_{24}$} at 230 310

\pinlabel \textcolor{red}{$c_{25}$} at 637 320
\pinlabel \textcolor{magenta}{$c_{26}$} at 622 349
\pinlabel \textcolor{red}{$c_{27}$} at 637 375
\pinlabel \textcolor{magenta}{$c_{28}$} at 622 402
\pinlabel \textcolor{red}{$c_{29}$} at 637 430

\pinlabel \textcolor{magenta}{$c_{30}$} at 678 439

\pinlabel \textcolor{red}{$c_{31}$} at 723 430
\pinlabel \textcolor{magenta}{$c_{32}$} at 740 400
\pinlabel \textcolor{red}{$c_{33}$} at 723 375
\pinlabel \textcolor{magenta}{$c_{34}$} at 740 347
\pinlabel \textcolor{red}{$c_{35}$} at 723 320

\pinlabel \textcolor{red}{$c_{36}$} at 637 155
\pinlabel \textcolor{magenta}{$c_{37}$} at 622 180
\pinlabel \textcolor{red}{$c_{38}$} at 637 208
\pinlabel \textcolor{magenta}{$c_{39}$} at 622 237
\pinlabel \textcolor{red}{$c_{40}$} at 637 265

\pinlabel \textcolor{magenta}{$c_{41}$} at 678 141

\pinlabel \textcolor{red}{$c_{42}$} at 723 155
\pinlabel \textcolor{magenta}{$c_{43}$} at 740 180
\pinlabel \textcolor{red}{$c_{44}$} at 723 208
\pinlabel \textcolor{magenta}{$c_{45}$} at 740 237
\pinlabel \textcolor{red}{$c_{46}$} at 723 265

\pinlabel \textcolor{magenta}{$c_{47}$} at 302 190
\pinlabel \textcolor{magenta}{$c_{48}$} at 418 400

\endlabellist
\centering
\includegraphics[scale=.45]{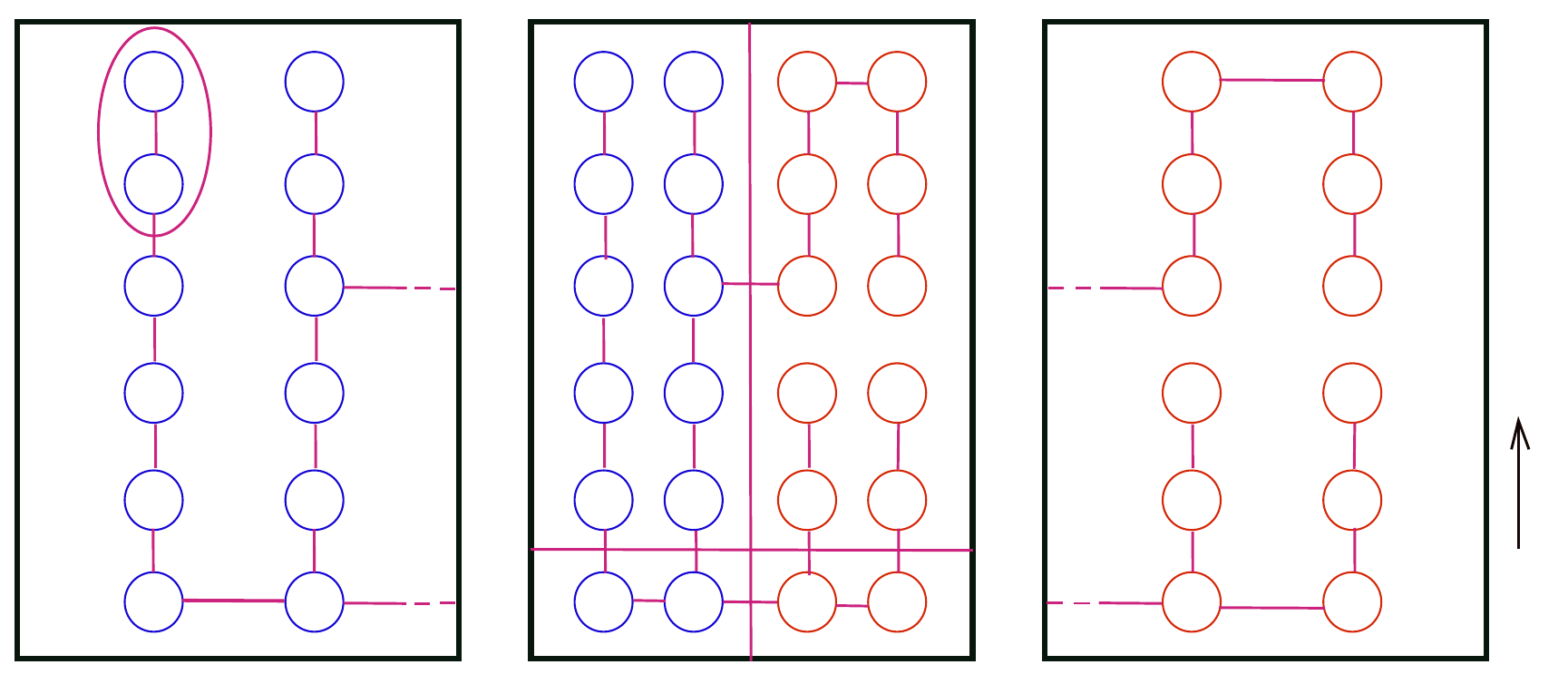}
\caption{The stabilization sequence we use for the case $k=6$ and $p=3$.} 
\label{fig:stab-sequence} \end{figure}
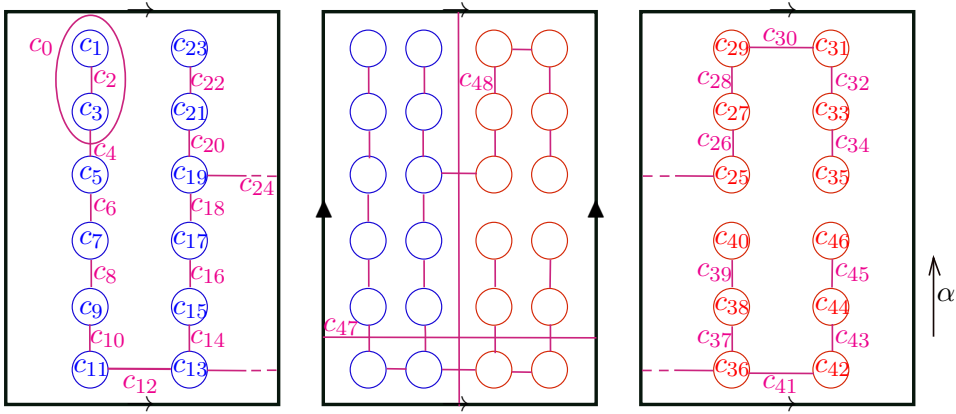

\subsection{Step 4: conclusion} 
By Proposition \ref{prop:fixed-set}, $\what\sigma(\alpha)$, $\what\sigma(\alpha^2)$ and $\what\sigma(\alpha^p)$ all have the same fixed set, which is a circle $c\subset \what X$. The centralizers $C(\alpha^2)$ and $C(\alpha^p)$ preserve $c$. By Proposition \ref{prop:centralizer}, $\Mod(\Sigma,*)=\langle C(\alpha^2),C(\alpha^p)\rangle$, so $\what\sigma\big(\Mod(\Sigma,*)\big)$ preserves $c$. This contradicts the fact that $\widehat\sigma(\pi_1(\Sigma_g))$ acts as the deck group, which as a properly discontinuous action  does not preserve any compact set.

\vspace{1in}
\bibliographystyle{alpha}
\bibliography{citing}

\end{document}